\documentclass[11pt]{article}
\usepackage{graphicx}
\usepackage{latexsym}
\usepackage{amsmath}
\usepackage{amssymb}
\usepackage{amsfonts}
\usepackage[all]{xy}
\usepackage{makeidx}
\input{epsf.tex}
\input xy
\xyoption{all}
\usepackage{enumerate}
\usepackage{natbib}
\usepackage{float}
\def\cite{\citet}
\newtheorem{lemma}{Lemma}
\newtheorem{proposition}{Proposition}
\newtheorem{theorem}{Theorem}

\input xy

\xyoption{all}

\newcommand{\qed}{\hfill \mbox{\raggedright \rule{.07in}{.1in}}}
\newenvironment{proof}{\vspace{1ex}\noindent{\bf Proof}\hspace{0.5em}}
	{\hfill\qed\\[5pt]}%
\begin{document} 
\title{Strategic Equilibria in Queues \\with Dynamic Service Rate and Full Information\footnote{This research has been co-financed by the European Union (European Social Fund - ESF) and Greek national funds through the Operational Program ''Education and Lifelong Learning'' of the National Strategic Reference Framework (NSRF)-Research Funding Program: Thales-Athens University of Economics and Business-New Methods in the Analysis of Market Competition: Oligopoly, Networks and Regulation.}}
\author{Apostolos Burnetas\footnote{Corresponding Author} and Yiannis Dimitrakopoulos\\
  Department of Mathematics\\
  National and Kapodistrian University of Athens\\
  Panepistemiopolis, Greece 15784\\
  email: aburnetas@math.uoa.gr, dimgiannhs@math.uoa.gr} \date\today

\maketitle

\begin{abstract}
We consider the problem of individual customer equilibrium for joining a single server Markovian queue, with state-dependent, nondecreasing service rates. Customers are homogeneous and make join/balk decisions to maximize their expected net benefit, having full information on the current queue length upon arrival. We develop a system of linear equations for the computation of the expected delay function of a customer for any symmetric joining strategy and derive necessary and sufficient equilibrium conditions. For pure and mixed threshold symmetric equilibrium strategies we establish a monotonicity property of the delay function, which leads to a finite algorithm for identifying equilibria. We also characterize the equilibrium strategies when the service rate policy is itself of threshold type between a low and a high service rate. Finally in a set of numerical experiments we show that in general there exist multiple symmetric threshold equilibrium strategies. 

\end{abstract}
% The body of the paper start here:

\section{Introduction}
\label{sec-Introduction}
%\noindent

The dynamic service  adjustment of service rate, when technically feasible, is a useful tool for control of congestion and customer delays in a queueing system. There exist several ways to implement service rate control, for example server activation and idling in multi-server systems or service speed adjustments in cases where the service is provided by a machine. In general, the objective of a service control policy is to achieve a balance between excessive customer delays when service speed is low and increased operational and maintenance costs when it is high. 

In systems where customers behave strategically and decide whether to join the system or not based on their anticipated delay, the effect of service rate variations becomes more complicated, because varying the service rate affects congestion and delays both directly, as well as indirectly since it also affects the arrival rate. 

An additional factor that affects performance in the  strategic customer framework is the level of information available to arriving customers. In most systems with physical customer presence  such as banks, public services, health care providers, etc., arriving customers can observe the congestion level before making any joining decisions.  The possibility for potential customers to observe the queue length, in addition to being aware of the  service policy, may have a significant effect on the effective arrival stream. Since a customer's delay depends on the decisions
of other customers, the incoming stream is the result of an equilibrium strategy of the game formulated by the customer strategic behavior.

The present paper proposes a direct model to analyze the effect of varying service rate on customer equilibrium behavior under full information. More specifically, we analyze the customer equilibrium behavior for joining a single server Markovian queue where the service rate is non-decreasing in the number of customers in the system. We analyze the observable case of this model where arriving customers are aware for both the service policy and the current queue length upon arrival.

Following the framework of \cite{naor}, we first prove that mixed strategies under which customers always join with positive probability cannot be equilibria. For strategies such that customers balk with probability one at a finite state, we show that the corresponding Markov chain for the number of customers in the system has a unique finite recurrent class. By restricting attention to customer strategies that satisfy the equilibrium conditions only for states in the recurrent class, on the one hand the analysis is considerably simplified, and, on the other hand, we do not lose much of the generality of the model, since in the infinite horizon framework the contribution of the transient states to the long-run average costs and profits is zero. 

In this paper, we restrict the analysis to the class of  threshold strategies, i.e. strategies where arriving customers 
join if and only if the number of customers they find in the system is below a threshold value, possibly randomizing at the threshold state. Strategies of this type  are intuitive and further simplify the performance as well as the equilibrium analysis.  Moreover, a thorough numerical investigation did not reveal any equilibrium strategies of non-threshold structure. 

Under the class of threshold strategies, we prove that an entering customer's expected delay is increasing in the number of customers already present. Although this property is obvious when the service rate is constant, under the increasing service rate assumption it becomes less intuitive; indeed, when more customers are present, the server works at higher rate thus decreasing the congestion faster. Using coupling arguments we show that the monotonicity nevertheless holds. 
This property introduces a significant simplification on the equations of equilibrium  and allows the construction of a finite algorithm to identify all pure threshold strategies.

 An implication of this property is that the possible threshold values for an equilibrium strategy are restricted to a finite range. As a result, we can construct an efficient search algorithm to identify all equilibrium strategies of pure threshold type. Moreover, considering the case of a threshold service policy where the service rate is dynamically adjusted between a low and high value according to a service threshold, we show that there may be at most one pure threshold equilibrium strategy when strategy's threshold is below the service threshold, whereas, on the contrary, numerical experiments demonstrate that there may be multiple equilibria when we consider strategies with threshold above the service threshold value. 

The research area that analyzes the strategic behavior of customers in a queueing system and its implications on the system performance has experienced considerable growth in recent years. The seminal papers of \cite{naor} and \cite{edelsonhildebrand}, analyze the simple model of an M/M/1 queue for the observable and unobservable case, respectively. Numerous variations of the original models under various levels of information have been studied since. The monographs of \cite{hassinhaviv}, \cite{stidham_book} and \cite{hassin-book2016} provide extensive reviews of models and results on queueing games and the economic analysis of service systems. Among the models that have been developed and analyzed, many include some varying service rate characteristics indirectly, while  the state information  may or may not be available to potential customers.
%
% In observable models, where customers obtain full or partial
% information on system congestion or anticipated delay, a varying
% service rate policy may have counterintuitive effects. 
\cite{armonymaglaras}
analyze the impact of announcing anticipated delays and providing a
call back option on customer joining behavior, in a call center with
two service modes. \cite{debo_2011} consider the strategic behavior of consumers  who join in an observable single server Markovian system with variable service rate and buy a product with varying quality. They show that the customer equilibrium strategy is of threshold type with threshold depending on product quality, and, that under certain conditions, a high-quality firm may serve in a slower mode than a low-quality firm. In this case the service policy can be a valuable signaling device for a high-quality firm. Moreover, \cite{guo_2011} studied models under partial information on service time distribution and they show that as the level of available information increases, more customers join. Thus, for a central planner it is better to give partial information to customers under individual welfare maximization whereas it is  more beneficial to reveal full information in profit maximization. In  \cite{hassinkoshman_2017}, the authors suggest a new model for issuing high-low delay announcements to customers who join a typical $M/M/1$ queue combined with a pricing polisy that charges a single price equal to the customer expected net benefit when the number of customers announced is below Naor's threshold. 
 
Other service systems with varying service rate characteristics are vacation queues where the server turns off and reactivates after a random time as in \cite{burnetaseconomou} and \cite{sunguo_2010} or the server resumes service after a fixed number of arrivals according to a threshold service policy as in \cite{vacation_hass,guo_hassin_vac2} and \cite{guoli_2013}. These papers consider customer equilibrium behavior with respect to several levels of information about the queue length and the state of the server with or without delay sensitive customers. In a make-to-stock production environment, \cite{li2_guo_song_2017} study customer strategic behavior on buying a product, where production is made under a bi-level threshold vacation policy, which engages production when the number of waiting customers reaches a certain level and ends it when the inventory level reaches a certain quantity. For this problem, the authors formulate a Stackelberg game between the production manager and potential customers and derive the customer equilibrium strategy, as well as the  optimal vacation policy for the firm. 

Another  facet of varying service rate policies, where the flexibility of increasing server speed results in more frequent departures, is to serve customers in batches instead instead of one by one. In this direction, \cite{economou_manou2013, manou_econ_kar_2014, manou_canb_kar_2017} and \cite{bountalieconomou2017} have analyzed the customer balking behavior when the service facility can serve a group of waiting customers either at once, i.e. a clearing system, or according to a threshold based batch service policy. In the latter case, the server works on a batch of a fixed size and starts working when this size is filled. The clearing system has been studied in \cite{economou_manou2013,manou_econ_kar_2014,manou_canb_kar_2017} where the system removes all present customers periodically and determined equilibrium strategies under a random environment and various levels of information. On the other hand, \cite{bountalieconomou2017} studied a model with batch services of a fixed size under two information scenarios. In the unobservable case the authors prove that there exist multiple equilibria, whereas in the observable case the customer behavior is affected by the balking behavior of future arrivals, and, thus dominant strategies are no longer available as in the single customer service.   

Finally, the unobservable case of the model in this paper is analyzed in \cite{Dim_Burnetas2011}, where customers are aware of the queue length upon arrival, but they all know the service policy induced by the service manager. It is shown that there exist at most three equilibrium strategies. A similar model, where the administrator employs a threshold-based staffing policy which activates or deactivates additional servers with respect to the system congestion, has been analyzed in \cite{guozhang_2013}. It is also shown that in general there exist multiple equilibrium strategies. In both papers the existence of multiple equilibria in general is due to the non monotonic behavior of customers expected sojourn time, and, as a result, the presence of both Avoid-the-Crowd and Follow-the-Crowd customer behavior in the unobservable case. On the other hand, in \cite{guozhang_2013}, revealing the service mode, i.e. whether additional servers are activated or not, results in a socially undesirable joining behavior where all customers join the costly system when all servers are on, and join the free system when some servers are deactivated. 

The contribution of this paper lies on the analysis of customers equilibrium behavior in a fully observable queueing Markovian system  under a general service policy with service rates being non-decreasing in the number of customers in the system where customers observe the queue length and, thus, also the service mode, upon arrival. Specifically, we show that under a general service policy of dynamically increasing the service rate as the queue length increases and pure threshold join/balk strategies, customer's expected delay is non-decreasing in the number of customers. The latter allows us to develop an efficient method to identify the equilibrium thresholds. Furthermore, we derive results on the existence of mixed threshold equilibrium strategies and relate them to the pure threshold case. Finally, we show using a special case of a threshold-based service rate policy that there are cases where there exist more than one equilibrium threshold strategies.

The rest of the paper is organized as follows. In Section \ref{sec-Model} we introduce the
 model and the corresponding customer strategic behavior problem.
In Section \ref{sec-Equilibrium} we present the equilibrium analysis of the model deriving necessary and sufficient conditions for equilibria and the expected waiting time of an arriving customer.
In Sections \ref{sec-pure} and \ref{sec-Mixed} we perform equilibrium analysis for the classes of pure and mixed  threshold strategies, respectively. In Section \ref{sec-thres} the equilibrium analysis is specialized to the case of a single-threshold service rate policy. Numerical experiments with respect to the service reward are presented in Section \ref{sec-comp}. Section \ref{sec-Summary} concludes.

\section{Model Description}\label{sec-Model}
\noindent We consider a single server Markovian queue under the FCFS
discipline, where potential customers arrive according to a Poisson
process with rate $\lambda$. The system administrator varies the service rate according to the number of customers present in the system at any time instant.  Specifically, the service policy is defined by a non-decreasing sequence of instantaneous service rates $\mu_n,;n\geq0$, where $n$ denotes the number of customers in the system. We assume that $\lim_{n\rightarrow\infty}\mu_n=M<\infty$. Finally, there are no service rate switching costs. 
 
Arriving customers are assumed identical and homogeneous. They observe the number of customers already present and are aware of the service policy. They make join decisions upon arrival, in order to maximize their expected net benefit, thus, they are risk neutral. Every customer who joins the system, receives a fixed reward $R\geq0$ upon service completion and incurs a waiting cost $C>0$ per time unit until departure, since she cannot renege after entering the system. 

In the observable model, arriving customers are aware of the total number of present customers in the system upon arrival, and, thus their join decisions depend on $n$. Since the join decisions of individual customers affect the system delay, and thus the benefit of all customers, the decision problem corresponds to a game among customers. We restrict attention to symmetric Nash equilibrium strategies. Specifically, given that potential customers are aware of the actual system state, an arriving customer has two pure strategies, either to join the system or balk. However, for the equilibrium analysis, mixed strategies also have to be considered. A mixed strategy is defined as a  probability vector $\underline{p}=(p_0,p_1,\ldots,p_n,\ldots),\;p_n\in[0,1]$, where $p_n$ denotes the join probability of a customer when there are $n$ customers present upon arrival. Let $\Pi=\{\underline{p}:p_n\in[0,1]\}=[0,1]^{\infty}$ be the set of mixed strategies. Under any strategy $p\in\Pi$, the Markov chain that describes the evolution of the number of customers in the system is a birth-and-death process with birth rates $\lambda_n=\lambda p_n$ for $n\geq0$ and death rates $\mu_n$, for $n\geq1$.
%Note that $p_n=0,\;\hbox{or}\;1$ refers to join or balk, respectively.

Consider an arriving customer who finds $n$ customers in the system upon arrival, and follows mixed strategy $\underline{q}$, whereas all other customers follow mixed strategy $\underline{p}$. Letting $W(n;\underline{p})$ be his/her expected sojourn time in the system, the tagged customer's expected net benefit, given by $U_n(\underline{q};\underline{p})$, can be expressed as
\begin{equation}\label{expnetbf}U_n(\underline{q};\underline{p})=q_n\left[R-CW(n;\underline{p})\right]=Cq_n\left[\tilde{R}-W(n;\underline{p})\right],\end{equation}
where $\tilde{R}=\frac{R}{C}$ expresses the relative importance of the service reward vs the cost of waiting.

Then, her best response to a strategy $\underline{p}$ is given by
\begin{equation}\label{bestrespfct} q^{B}_n(\underline{p})= \left\{\begin{array}{ll}0,&\mbox{if }  \tilde{R}-W(n;\underline{p})<0\\
\in[0,1],&\mbox{if }  \tilde{R}-W(n;\underline{p})=0\\1,&\mbox{if }  \tilde{R}-W(n;\underline{p})>0\end{array}\right.,
\end{equation}

and, a mixed strategy $\underline{p}^e$ is a symmetric Nash equilibrium strategy, if $U_n(\underline{p}^e,\underline{p}^e)\geq U_n(\underline{q},\underline{p}^e)$ for any strategy $\underline{q}\in\Pi$ and any state $n$, i.e. $\underline{p}^e$ is the best
response against itself, since if all customers agree to follow $\underline{p}^e$, no one can benefit
from changing it.

In the remainder of the paper, we refer to $W(n;\underline{p})$ as the waiting time or the delay function.

\section{Equilibrium Analysis} \label{sec-Equilibrium}
\noindent In this section, we consider the problem of individual equilibria. Let $\Omega$ be the set of symmetric Nash equilibrium strategies. Since $U_n(\underline{q};\underline{p})$ depends on $\underline{q}$ only through the probability $q_n$, from \eqref{expnetbf} and \eqref{bestrespfct} it follows that any equilibium strategy satisfies:
\begin{eqnarray*}
&&\hbox{If }\tilde{R}-W(n;\underline{p}^e)<0\Rightarrow q^{B}_n(\underline{p}^e)=0\Rightarrow p^e_n=0.\label{necequil0}\\
&&\hbox{If }\tilde{R}-W(n;\underline{p}^e)>0\Rightarrow q^{B}_n(\underline{p}^e)=1\Rightarrow p^e_n=1.\label{necequil1}\\
&&\hbox{If}\;p^e_n\in(0,1)\Rightarrow \tilde{R}-W(n;\underline{p}^e)=0.\label{necequilrand}
\end{eqnarray*}

From the above, we obtain that a mixed strategy $\underline{p}^e$ is equilibrium, i.e.,
$\underline{p}^e\in \Omega$, if and only if the following inequalities hold.
\begin{eqnarray}
&\tilde{R}-W(n;\underline{p}^e)\leq0,&\hbox{for all } n \hbox{ such that }p^e_n=0 .\label{suf_nec_equil0}\\
&\tilde{R}-W(n;\underline{p}^e)\geq0,&\hbox{for all } n \hbox{ such that }p^e_n=1.\label{suf_nec_equil1}\\
&\tilde{R}-W(n;\underline{p}^e)=0,&\hbox{for all } n \hbox{ such that }p^e_n\in(0,1).\label{suf_nec_equilrand}
\end{eqnarray}

We next develop further results that simplify these equilibrium conditions in the sense that for an equilibrium these inequalities must be valid only for a finite number of states. Specifically, we first show that for any mixed strategy there exists a finite state where joining is not optimal. Thus, the equilibrium analysis can be restricted  to strategies where balking is prescribed in a finite state. Furthermore, for strategies of this type the corresponding Markov chain has a single ergodic class, which implies that in steady state the equilibrium conditions in \eqref{suf_nec_equil0} - \eqref{suf_nec_equilrand} must be satisfied only for the ergodic states.

For a mixed strategy $\underline{p}=(p_0,p_1,\ldots)\in\Pi$, let 
\begin{equation}
\label{ne}n_{b}(\underline{p})=\min\{n:\;\;\tilde{R}-W(n;\underline{p})<0\},
\end{equation}
and
\begin{equation}
\label{no}n_0(\underline{p})=\min\{n:p_n=0\},
\end{equation}
denote the first state where the expected net benefit from joining is strictly negative and the first state where strategy $\underline{p}$ prescribes balking, respectively.

In Lemma \ref{first res}, we establish bounds on the waiting time function $W(n;\underline{p})$, which implies an ordering between $n_b$ and $n_0$ in equilibrium.

\begin{lemma}\label{first res}
\begin{itemize}
\item[1.]For any mixed strategy $\underline{p}\in\Pi$, the following hold:
\begin{itemize}
\item[i.] $\frac{n+1}{M}\leq W(n;\underline{p})\leq\frac{n+1}{\mu_1},$ for any $n$.
\item[ii.] $n_{b}(\underline{p})<\infty$.
\end{itemize}
\item[2.] For any equilibrium strategy $\underline{p}^e$, $n_0(\underline{p^e})\leq n_{b}(\underline{p^e})<\infty$.
\end{itemize}
\end{lemma}

\begin{proof}
\begin{itemize}
\item[1.]  Statement (i) is immediate from the definition of $W(n;\underline{p})$, since the bounds correspond to the expected waiting time of a tagged customer who has $n$ customers ahead, is last in queue, and the service rate is kept constant at its highest or lowest possible value, 

Now for any $n\geq\tilde{R}\;M$, it follows that $\tilde{R}-W(n;\underline{p})\leq -\frac{1}{M}<0$. 
Therefore, $n_b(\underline{p})\leq \lfloor\tilde{R}\;M\rfloor+1<\infty$.

\item[2.]
From 1(ii) and the equilibrium conditions \eqref{suf_nec_equil0}-\eqref{suf_nec_equilrand}, it follows that any $\underline{p}^e\in\Omega$ satisfies $p^e_n=0$ for all $n\geq n_b(\underline{p}^e)$. Therefore $n_0(\underline{p^e})\leq n_{b}(\underline{p^e})$.
\end{itemize}
%\ \\[-12pt]
\end{proof}
From Lemma \ref{first res}, it follows that $\Omega\subseteq\Pi_{M}$, where $\Pi_{M}=\{\underline{p}\in\Pi:\;\;n_0(\underline{p})<\infty\}$. Therefore, we  can restrict the analysis to strategies where balking first occurs at a finite state, i.e., the possibility of never balking is excluded.

Furthermore, we observe that under any strategy $\underline{p}\in\Pi_M$, the  corresponding Markov Chain of the number of customers in the system has a single ergodic class  $\{0,1,\ldots,n_0(\underline{p})\}$, while states $\{n_0(\underline{p})+1,\ldots\}$ are transient.

In general, for $\underline{p}^e\in\Omega$, the equilibrium conditions must be satisfied for all states, both ergodic and transient. However, under any strategy  $\underline{p}^e\in\Pi_M$, with probability $1$ only a finite number of customers will encounter a transient state upon arrival and their decisions do not affect the system performance and the expected customer net benefit in steady state. 

Therefore, in the steady state framework we may relax the conditions for equilibrium and demand that \eqref{suf_nec_equil0}-\eqref{suf_nec_equilrand} are satisfied by a strategy $\underline{p}\in\Pi_M$ only for states in the ergodic class. 

In this sense, we define the corresponding strategy set as 
\begin{equation*}\Omega_{RC}=\{\underline{p}\in\Pi_{M}:\;\hbox{equilibrium conditions in \eqref{suf_nec_equil0},\eqref{suf_nec_equil1} and \eqref{suf_nec_equilrand} hold},\;\hbox{ for any } n\leq n_0(\underline{p})\}.
\end{equation*}
and refer to strategies $\underline{p}\in\Omega_{RC}$ as recurrent class equilibria. It is immediate that $\Omega\subseteq\Omega_{RC}\subseteq\Pi_M$. 

Note that the strategies in $\Omega$, which in addition to the ergodic states satisfy the equilibrium conditions also for transient states, are defined as subgame perfect equilibria (SPE) in \cite{hassinhaviv} and  \cite{hassin-book2016}.

For recurrent-class equilibria $\underline{p}\in\Omega_{RC}$, the necessary and sufficient conditions are:
\begin{equation}
%\begin{displaymath}
\underline{p}\in\Omega_{RC}\;\hbox{if and only
if}\;\left\{\begin{array}[pos]{l}\tilde{R}-W(n_0(\underline{p});\underline{p})\leq0,\\ \tilde{R}-
W(n;\underline{p})\geq0,\; \forall n\leq
n_0(\underline{p})-1,\label{condequilRC}\\ \tilde{R}-W(n;\underline{p})=0,\;\forall n\leq
n_0(\underline{p})-1:p_n\in(0,1)\end{array}\right..
%\end{displaymath}
\end{equation}

Thus, the conditions for $p\in\Omega_{RC}$ require calculation of the waiting time $W(n;\underline{p})$ only for states in the recurrent class.

We now proceed to the computation of the delay function $W(n;\underline{p}),\;n=0,1,\ldots,n_0(\underline{p}),$ for a strategy $\underline{p}\in\Pi_M$. In this case, the system reduces in an $M/M/1$ queue with finite buffer size, $n_0=n_0(\underline{p})$, arrival rates $\lambda_n=\lambda p_n$ and service rates $\mu_n$ for $n=1,\ldots,n_0(\underline{p})$.

In order to compute $W(n;\underline{p})$, we generalize the waiting time definition and consider the function $W(n,m;\underline{p})$, which denotes the expected waiting time of a tagged customer already in the system, given that there are $n$ customers in front of him, $m$ customers in total in the system and all future arrivals follow mixed strategy $\underline{p}$. The dependence on $m$ is needed because the service rate is state dependent and may switch several times after a customer's entrance depending on the number of future entrances. Using this definition, the waiting time $W(n;\underline{p})$ of a joining customer may be expressed, as follows:
\begin{equation}\label{wtf}
W(n;\underline{p})=\left\{\begin{array}[pos]{ll}
W(n,n+1;\underline{p}),&n=0,1,\ldots,n_0-1\\
\frac{1}{\mu_{n_0+1}}+W(n_0-1,n_0);\underline{p}),&n=n_0
\end{array}\right..
\end{equation}

For the second branch in \eqref{wtf}, we note that, although strategy $\underline{p}$ prescribes balking for $n=n_0$, in order to characterize the  equilibrium, we need an expression for $W(n_0;\underline{p})$, i.e. the waiting time of a tagged customer who finds $n_0$ customers in the system upon arrival, and nevertheless joins. In this case, all future arrivals will balk until the next departure, and thus, the customer in service has a residual service time exponentially distributed with $\mu_{n_0+1}$.

From first-step analysis, we can derive equations for the generalized waiting time $W(n,m;\underline{p})$, as follows.

Assume that all customers follow mixed strategy $\underline{p}\in\Pi_M$ and let $n_0=n_0(\underline{p}$. Consider a tagged customer and define a new Markov Chain that describes the tagged customer's position until his departure. The state of this process is the pair $(n,m)$ as defined above. The transition diagram is presented in Figure \ref{statediagram}. Note that, the transition rates are noted above each arrow, and {\it{(D)}} refers to departure.

\begin{figure}[H]
\centering
$$ \def\labelstyle{\scriptstyle}
 \xymatrix @C+0.5pc  {
&&&&&&&n_0-1,n_0\ar[dl]_{\mu_{n_0}}\\
&&&&&&n_0-2,n_0-1\ar[dl]_{\mu_{n_0-1}}\ar[r]^{\lambda p_{n_0-1}}&n_0-2,n_0\ar[dl]_{\mu_{n_0}}\\
&&&&&&\vdots&\vdots\\
&&&&m,m+1\ar[dl]_{\mu_{m+1}}\ar[r]^{\lambda p_{m+1}}&\cdots\ar[r]^{\lambda p_{n_0-2}}&m,n_0-1\ar[dl]_{\mu_{n_0-1}}\ar[r]^{\lambda p_{n_0-1}}&m,n_0\ar[dl]_{\mu_{n_0}}\\
&&&m-1,m\ar[dl]_{\mu_m}\ar[r]^{\lambda p_{m}}&m-1,m+1\ar[dl]_{\mu_{m+1}}\ar[r]^{\lambda
p_{m+1}}&\cdots\ar[r]^{\lambda p_{n_0-2}}&m-1,n_0-1\ar[dl]_{\mu_{n_0-1}}\ar[r]^{\lambda
p_{n_0-1}}&m-1,n_0\ar[dl]_{\mu_{n_0}}\\
&&&\vdots&\vdots&&\vdots&\vdots\\
&1,2\ar[dl]_{\mu_2}\ar[r]^{\lambda p_2}&\cdots\ar[r]^{\lambda p_{m-1}}&1,m\ar[dl]_{\mu_m}\ar[r]^{\lambda
p_{m}}&1,m+1\ar[dl]_{\mu_{m+1}}\ar[r]^{\lambda
p_{m+1}}&\cdots\ar[r]^{\lambda p_{n_0-2}}&1,n_0-1\ar[dl]_{\mu_{n_0-1}}\ar[r]^{\lambda
p_{n_0-1}}&1,n_0\ar[dl]_{\mu_{n_0}}\\
0,1\ar[drrr]_{\mu_1}\ar[r]^{\lambda
p_1}&0,2\ar[drr]_{\mu_2}\ar[r]^{\lambda p_2}&\cdots\ar[r]^{\lambda p_{m-1}}&0,m\ar[d]_{\mu_m}\ar[r]^{\lambda p_m}&0,m+1\ar[dl]_{\mu_{m+1}}\ar[r]^{\lambda p_{m+1}}&\cdots\ar[r]^{\lambda p_{n_0-2}}&0,n_0-1\ar[dlll]_{\mu_{n_0-1}}\ar[r]^{\lambda
p_{n_0-1}}&0,n_0\ar[dllll]_{\mu_{n_0}}\\
&&&(D)&&&&}
%\caption{State transition diagram of customers flow}
$$
\caption{State Transition Diagram}\label{statediagram}
\end{figure}
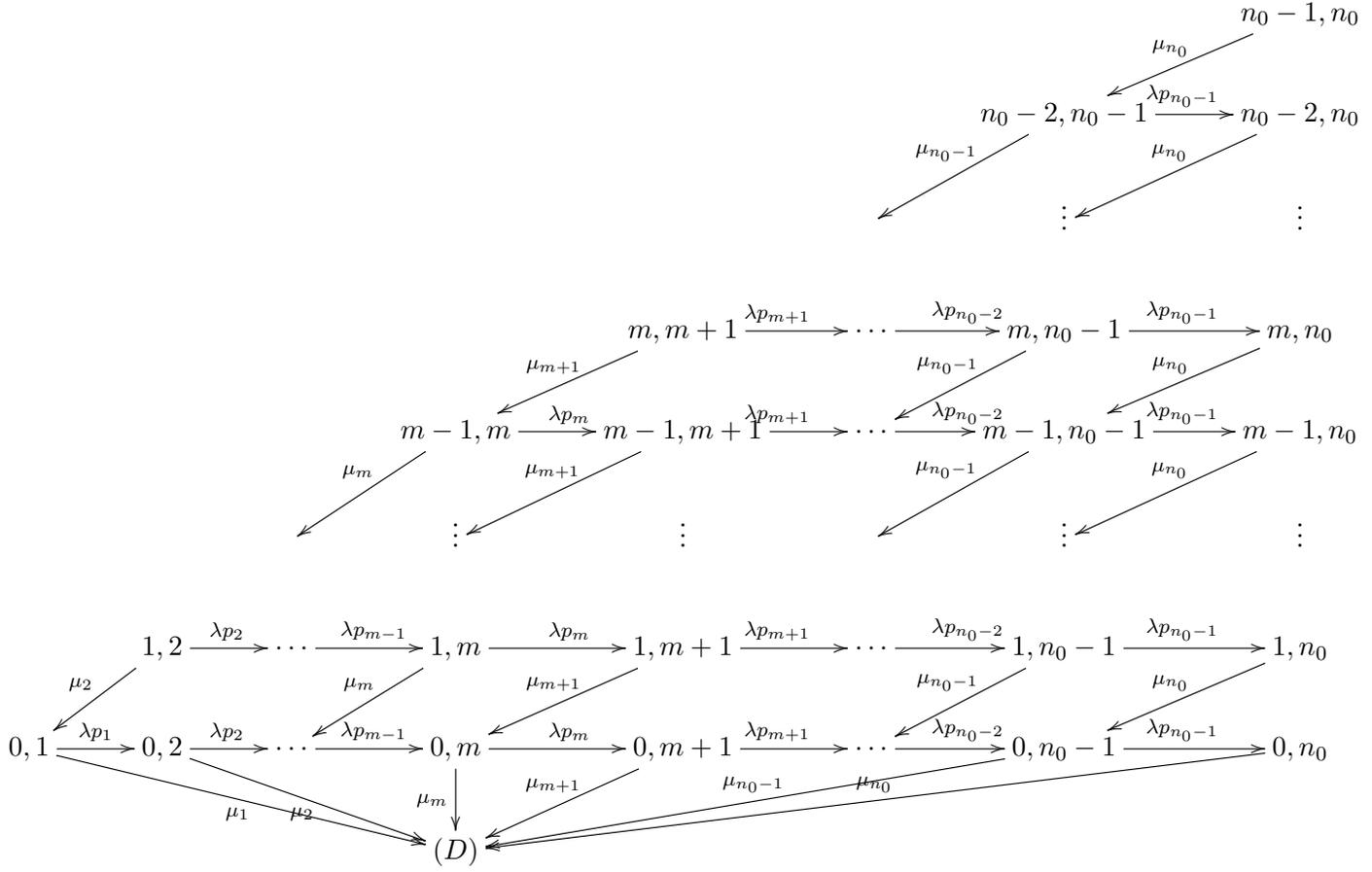

From the state transition diagram and first-step analysis, we can derive $W(n,m;\underline{p})$ as the solution of the following system of linear equations. We omit the dependence on the strategy $\underline{p}$.

For $n=0$:
%\begin{equation}
%\label{w0n0}W(0,n_0)=\frac{1}{\mu_{n_0}},
%\end{equation}
\begin{equation}
\label{w0m}W(0,m)=\frac{1}{\lambda p_m+\mu_m}+\frac{\lambda
p_m}{\lambda p_m+\mu_m}W(0,m+1),\hbox{ for } 1\leq m\leq n_0.
\end{equation}
%\begin{equation}
%\label{w0mlow}W(0,m)=\frac{1}{\lambda p_m+\mu_l}+\frac{\lambda
%p_m}{\lambda p_m+\mu_l}W(0,m+1),\hbox{ for }1\leq m\leq T.
%\end{equation}

Similarly, for $1\leq n\leq n_0-1$:
%\begin{equation}
%\label{wnn0}W(n,n_0)=\frac{1}{\mu_{n_0}}+W(n-1,n_0-1),
%\end{equation}
\begin{eqnarray}
\label{wnm}W(n,m)=\frac{1}{\lambda p_m+\mu_m}+\frac{\lambda
p_m}{\lambda p_m+\mu_m}W(n,m+1)+\frac{\mu_m}{\lambda p_m+\mu_m}W(n-1,m-1),\\\nonumber\hbox{ for }n+1\leq m\leq n_0.
\end{eqnarray}
Note that, for $m=n_0$ the above equations can be simplified to 
\begin{equation*}
W(0,n_0)=\frac{1}{\mu_{n_0}}, \hbox{ and } W(n,n_0)=\frac{1}{\mu_{n_0}}+W(n-1,n_0-1),
\end{equation*}
since $p_{n_0}=0$. 

The waiting time function $W(n;\underline{p})$ follows from the solution of  \eqref{w0m}, \eqref{wnm}, as well as \eqref{wtf}.
%\begin{equation}\label{expwtf_2}W(n;\underline{p})=\left\{\begin{array}[pos]{ll}
%W(n,n+1),&0\leq n\leq n_0(\underline{p})-1\\
%\frac{1}{\mu_{n_0+1}}+W(n_0-1,n_0),& n=n_0(\underline{p})
%\end{array}\right.
%\end{equation}

%with $n_0(\underline{p})>T$, 

\section{Pure threshold strategies}\label{sec-pure}
In this section, we restrict attention to pure threshold strategies, i.e. strategies $\underline{p}\in\Pi_M$, such that $p_n=1$, for any $n\leq n_0(\underline{p})-1$. Under a pure threshold strategy, an arriving customer enters the system if and only if he/she finds at most $n_0-1$ customers already present in the system upon arrival.

Although considering pure threshold strategies is a restriction of the equilibrium class as discussed in the previous section, such equilibria, when they exist, have some appealing properties. They are easy to describe, since only $n_0$ is required, and, more importantly, the equilibrium analysis is considerably simplified due to a monotonicity property proven below. We have performed an extensive numerical search in a large range of parameter values which did not locate any equilibrium strategies that violated the threshold structure. 

Since a pure threshold strategy is uniquely determined by the threshold value $n_0$, the class of pure threshold strategies is equivalent to the set of positive integers. Specifically, let $\tilde{\underline{p}}(n_0)$ be the pure threshold strategy which is determined by $n_0\in N$, thus $\tilde{p}_n=1$ for any $n\leq n_0-1$ and $\Pi_{TH}=\{\tilde{\underline{p}}(n_0):\; n_0=1,2,\ldots\}$ be the set of pure threshold strategies. Then $\Omega_{TH}=\Omega_{RC}\cap\Pi_{TH}$, is defined as the set of pure threshold strategies that are symmetric equilibria in the recurrent class. 

For $\underline{p}\in\Omega_{TH}$ the conditions for equilibrium given in \eqref{condequilRC} simplify to the following:
\begin{equation}
\label{condequilpthr_V1}\tilde{\underline{p}}(n_0)\in\Omega_{TH}\;\hbox{if and only
if}\;\left\{\begin{array}[pos]{l} \tilde{R}-
W(n;\underline{p})\geq0,\; n=0,1,\ldots,
n_0-1\\ \tilde{R}-W(n_0;\underline{p})\leq0\end{array}\right.,
\end{equation}
since $ \tilde{p}_n=1$ for all $n\leq n_0-1$.

From Lemma \ref{first res}(i) and the equilibrium conditions given in \eqref{condequilpthr_V1}, we see that in order for $\tilde{\underline{p}}(n_0)\in\Omega_{TH}$, the threshold $n_0$ must lie in the interval,
\begin{equation}
\label{neccondequilpthr}
\left(\tilde{R}-\frac{1}{M}\right)\mu_1\leq n_0\leq \tilde{R}\;M.
\end{equation}

It follows that the number of pure threshold equilibrium strategies is finite. In particular, there exists a finite algorithm which can identify all the pure threshold equilibrium strategies by checking the equilibrium conditions for each positive integer $n_0$ lying in the interval given by \eqref{neccondequilpthr}. For any given $n_0$, this algorithm first solves the finite linear system of equations in \eqref{w0m}, \eqref{wnm} and then verifies the inequalities in \eqref{condequilpthr_V1}. Depending on the values of the parameters, the range in \eqref{neccondequilpthr} can be significantly large, however the number of the pure threshold strategies that must be checked for equilibria is finite. 

Finally, in Proposition \ref{st_induction} below, we prove that for any pure threshold joining strategy $\underline{\tilde{p}}(n_0)$ the waiting time $W(n;\underline{\tilde{p}}(n_0))$ is non-decreasing in $n$, since the generalized waiting time $W(n,n+1;\tilde{\underline{p}}(n_0))$ is non-decreasing in $n$. This leads to a significant simplification  of the equilibrium conditions. 

In view of \eqref{wtf}, to show this monotonic behavior, it is sufficient to show that for any pure threshold strategy $\underline{\tilde{p}}(n_0)$:
\begin{equation*}
W(n-1,n)\leq W(n,n+1),\hbox{ for }n\leq n_0-1.
\end{equation*}  
%\begin{proposition}\label{ewtmon}
%Given a pure threshold customer joining strategy $\underline{\tilde{p}}(n_0)$, the expected waiting time $W(n;\tilde{\underline{p}}(n_0))$ is non-decreasing in $n$ for any $n\leq n_0$.  
%\end{proposition}   
%
%\begin{proof}
%For any pure threshold customer strategy $\underline{\tilde{p}}(n_0)$ defined by an arbitrary positive integer $n_0>T$, the monotonicity of $W(n;\underline{\tilde{p}}(n_0))$ is equivalent to  the monotonicity of the generalized expected waiting time $W(n,n+1;\underline{\tilde{p}}(n_0))$ from \eqref{wtf}. 
%
%Therefore for a given pure threshold strategy $\underline{\tilde{p}}(n_0)\in\Pi_{TH}$, it is sufficient to show that 
%\begin{equation*}
%W(n-1,n)\leq W(n,n+1),\hbox{ for }n\leq n_0-1.
%\end{equation*}  
We prove this by a coupling argument. 

Specifically, we consider two $M/M/1$ queueing systems, denoted with $A$ and $B$. System $A$ has $n$ customers in total labeled from $1$ to $n$ where customer $1$ is in service, whereas system $B$ has $n+1$ customers in total labeled from $0$ to $n$ and serves customer $0$, respectively. Without loss of generality, we assume that customer $0$ just entered service in system $B$. From this time instance, denoted by $t=0$, and onward, we couple both systems on the arrival times of future customers $n+1,n+2,\dots$, as well as on the service process, as follows. Let $S_j$ denote the service requirements of customer $j$ for $j=1,2,\dots,n$. We assume that $S_j$ are i.i.d. random variables exponentially distributed with mean $1$. The service requirements of customer $1,2,\ldots,n$ are coupled in the two systems. In this framework, the service rate $\mu(t)$ corresponds to the instantaneous rate of decrease of $S$. Note that, if $\mu(t)$ is constant and equal to $\mu$ then the service time of any customer $j$ follows an exponential distribution with parameter $\mu$.  

For this coupling scheme, we show in Proposition \ref{st_induction} that $T_A(n)\leq T_B(n)$ with probability $1$, where $T_A(n), T_B(n)$ denote the sojourn time of customer $n$ in systems $A$ and $B$, respectively. By taking expectations, we then derive that $W(n-1,n)\leq W(n,n+1)$ for any $n\leq n_0-1$.  
%\end{proof}

\begin{proposition}\label{st_induction}
For any pure threshold strategy with threshold $n_0$ and any service policy with a non-decreasing service rate in the number of customers in the system, the following holds with probability $1$:
\begin{equation}\label{st_ineq} T_A(j)\leq T_B(j),\hbox{ for any }j=1,2,\ldots,n 
\end{equation}
\end{proposition}
 
\begin{proof}
We consider the coupling scheme described before and  assume that the adopted service policy in both systems prescribes a non-decreasing service rate with respect to the number of customers in the system. 

In the proof of the Proposition we will only consider sample paths in which no two events in the same system may occur simultaneously. Because all associated random variables are continuous, the excluded sample paths have probability zero.

We prove \eqref{st_ineq} by induction in $j$. For $j=1$, we must prove that $T_A(1)\leq T_B(1)$. We consider the following cases:
\begin{itemize}
\item[A1.] If $T_A(1)\leq T_B(0)$, then the result follows readily.

\item[A2.] Assume that $T_A(1)> T_B(0)$. Then at time instant $t=T_B(0)$, customer $1$ is still in service in system $A$, whereas in $B$ he starts his service. Therefore,
\begin{eqnarray}
N_A(T_B(0))&=&n+J_A(T_B(0)),\\
N_B(T_B(0))&=&n+J_B(T_B(0)),
\end{eqnarray}
where $N_M(t)$ and $J_M(t)$ denote the number of customers present in system $M$ at time $t$, and the number of customers who joined in $M$ in $[0,t]$, for $M=A,\;B$. We assume that $N_M(t)$ have right-continuous sample paths, e.g. if $t_D$ is a departure time from system $M$, then $N_M(t_D)$ denotes the number of customers in $M$ immediately after this departure. This assumption implies that $J_M(t)$ also has right-continuous sample paths. 

In order to prove that $T_A(1)\leq T_B(1)$, we first show that $N_A(T_B(0))\geq N_B(T_B(0))$. This is next shown to imply that the server in $A$ will work faster than in $B$ until the next departure.

We show the inequality by contradiction. Assume that
\begin{equation*}
N_A(T_B(0))< N_B(T_B(0)),\hbox{ or equivalently, }J_A(T_B(0))< J_B(T_B(0)).
\end{equation*} 

This means that at some earlier time $t_0<T_B(0)$, $A$ had reached the threshold $n_0$ and $B$ could still accept customers, i.e. $N_A(t_0)=n_0, N_B(t_0)<n_0$, so that a subsequent arrival joined $B$ but not $A$. 

However, the number of joining customers in $A$ in $[t_0,T_B(0))$ must be equal to $0$ since there is no departure from $A$ in $[t_0,T_B(0))$ , and, thus all future arrivals in $A$ until $T_B(0)$ are lost. Therefore, $N_A(T_B(0))=n_0\geq N_B(T_B(0))$, which is a contradiction. Therefore, $N_A(T_B(0))\geq N_B(T_B(0))$. 

Since $N_A(T_B(0))\geq N_B(T_B(0))$ and the service speed is non-decreasing in the number of customers, it follows that $\mu_A(T_B(0))\geq\mu_B(T_B(0))$ where $\mu_M$ refers to the service rate employed in system $M$ at time instant $t$. In addition, since future arrivals are coupled, the relationship between $N_A(t)$ and $N_B(t)$ at any time $t$  will not change until the next departure. Therefore $\mu_A(t)\geq\mu_B(t)$ for any $t$ until the next departure. Since the service requirements of customer $1$ are coupled in the two systems, and at $t=T_B(0)$ customer$-1$ in $A$ has already finished some part of his requirements, whereas the same customer in $B$ starts his service, it follows that the next departure will be from $A$, i.e. $T_A(1)<T_B(1)$. Therefore, \eqref{st_ineq} is true for $j=1$.
\end{itemize}

We next assume that \eqref{st_ineq} holds for any $i=1,\ldots,j$, and prove that it is true for $j+1$. 

From the induction hypothesis, it follows that $D_A(t)\geq D_B(t)-1$ for any $t\in[0,T_B(j)]$, where $D_A(t), D_B(t)$ refer to the number of departures from systems $A$ and $B$, respectively, at time instant $t$. Indeed, when any customer $i\leq j$ departs from system $B$, i.e., $D_B(T_B(i))=i+1$, this customer has already departed from system $A$, i.e., $D_A(T_B(i))\geq i$.

We consider the following cases.
\begin{itemize} 
\item[B1.] If $T_A(j+1)\leq T_B(j)$, then $T_A(j+1)\leq T_B(j)<T_B(j+1)$ and the result follows readily, since the queue discipline is FCFS.

\item[B2.] Assume that $T_A(j+1)> T_B(j)$, i.e. $T_A(j)\leq T_B(j)<T_A(j+1)$, thus at time instant $t=T_B(j)$, customer $(j+1)$ is still in service in system $A$, whereas the same customer in $B$ starts service at this instant. Therefore,
\begin{eqnarray}
\label{NA1}N_A(T_B(j))=n-j+J_A(T_B(j)),&\\
\label{NB1}N_B(T_B(j))=n+1-(j+1)+J_B(T_B(j))=&n-j+J_B(T_B(j)).
\end{eqnarray}

Similarly to the initial step of the induction for $j=1$, we will first show that \begin{equation*}
N_A(T_B(j))\geq N_B(T_B(j)).
\end{equation*} 

Assume that $N_A(T_B(j))< N_B(T_B(j))$, or equivalently $J_A(T_B(j))< J_B(T_B(j))$. Then in at least one instant before $T_B(j)$, system $A$ was full, $B$ had empty space and a customer joined $B$. Let $t_0<T_B(j)$ be the last time before $T_B(j)$ that this event occurred. Thus,  $N_A(t_0)=n_0$ and $N_B(t_0)=n_0-k$, for some $k\geq0$.  

Considering the time interval $(t_0,T_B(j)]$ we obtain the following:
\begin{eqnarray}
\label{NA2}N_A(T_B(j))=&N_A(t_0)+\tilde{J}_A((t_0,T_B(j)])-\tilde{D}_A((t_0,T_B(j)])=&n_0+\tilde{J}_A-\tilde{D}_A,\\
\label{NB2}N_B(T_B(j))=&N_B(t_0)+\tilde{J}_B((t_0,T_B(j)])-\tilde{D}_B((t_0,T_B(j)])=&n_0-k+\tilde{J}_B-\tilde{D}_B,
\end{eqnarray}
where $\tilde{J}_{M}:=\tilde{J}_{M}((t_0,T_B(j)])$ is the number of joining customers and $\tilde{D}_{M}:=\tilde{D}_{M}((t_0,T_B(j)])$ is the corresponding number of departures from system $M=A$ or $B$, respectively, in the interval $(t_0, T_B(j)]$. 

Solving \eqref{NA2} with respect to $\tilde{D}_A$, it follows that
\begin{equation}
\label{DAineq}\tilde{D}_A=n_0-N_A(T_B(j))+\tilde{J}_A>n_0-N_B(T_B(j))+\tilde{J}_A,
\end{equation}  
since $N_A(T_B(j))<N_B(T_B(j))$.

Thus, from \eqref{NB2},
\begin{equation}
\label{DAineq2}\tilde{D}_A>n_0-(n_0-k)-\tilde{J}_B+\tilde{D}_B+\tilde{J}_A=k+\tilde{D}_B+\tilde{J}_A-\tilde{J}_B.
\end{equation}
However,
\begin{eqnarray}
\label{DA1}\tilde{D}_A=&D_A(T_B(j))-D_A(t_0)=&j-D_A(t_0)\\
\label{DA2}\tilde{D}_B=&D_B(T_B(j))-D_B(t_0)=&j+1-D_B(t_0),
\end{eqnarray}

Substituting \eqref{DA1},\eqref{DA2} into \eqref{DAineq2}, we obtain  
\begin{eqnarray}
\nonumber j-D_A(t_0)>j+1-D_B(t_0)+k+\tilde{J}_A-\tilde{J}_B&\Leftrightarrow &\\\label{DAineq3} D_A(t_0)<D_B(t_0)-k-1+\tilde{J}_B-\tilde{J}_A&&.
\end{eqnarray}

From the induction hypothesis it follows that 
\begin{equation*}
D_B(t_0)-1\leq D_A(t_0)<D_B(t_0)-k-1+\tilde{J}_B-\tilde{J}_A, 
\end{equation*}
therefore
\begin{equation}
\label{Jineq} \tilde{J}_B>\tilde{J}_A+k,\hbox{ with }k\geq0.
\end{equation}

However, in the interval $(t_0,T_B(j)]$ there was never an instant where a customer joined system $B$ and not $A$, because $t_0$ is the last instant before $T_B(j)$ that this event happened. Therefore $\tilde{J}_B\leq\tilde{J}_A$ and this is a contradiction. Thus, $N_A(T_B(j))\geq N_B(T_B(j))$. 

Since $N_A(T_B(j))\geq N_B(T_B(j))$, it follows that $\mu_A(T_B(j))\geq\mu_B(T_B(j)).$ Following a completely similar argument as in the case for $j=0$, we obtain that \begin{equation}
\mu_A(t)\geq\mu_B(t),\;t\in[T_B(j),T_A(j+1)],
\end{equation}
and, since customer $j+1$ has already started service in $A$ at $t=T_B(j)$, he will depart sooner than customer $j+1$ in $B$.
\end{itemize}

Therefore, \eqref{st_ineq} is also true for $j+1$, and the proof is complete.
\end{proof} 

The monotonicity of $W(n;\underline{\tilde{p}}(n_0))$, allows us to further simplify the equilibrium conditions given in \eqref{condequilpthr_V1} as follows. Since $W(n;\tilde{\underline{p}}(n_0))$ is non-decreasing in $n$ for any $n\leq n_0-1$, 
\begin{equation*}
W(n;\underline{\tilde{p}}(n_0))\leq\tilde{R}\hbox{ for all }n\leq n_0-1 \hbox{ if and only if } W(n_0-1;\underline{\tilde{p}}(n_0))\leq\tilde{R},
\end{equation*}
and, thus,
\begin{equation}
\label{condequilpthrsvr_V2}\tilde{\underline{p}}(n_0)\in\Omega_{TH}\;\hbox{if and only
if}\;\left\{\begin{array}[pos]{l}\tilde{R}-W(n_0;\tilde{\underline{p}}(n_0))\leq0\\ \tilde{R}-
W(n_0-1;\tilde{\underline{p}}(n_0))\geq0\end{array}\right..
\end{equation}
The latter corresponds to a simple condition that characterizes pure threshold strategies. This condition is stated in the following theorem.

\begin{theorem}\label{condequilpthrsvr}
Given a pure threshold strategy $\underline{\tilde{p}}(n_0)$,
\begin{eqnarray}
\nonumber \tilde{\underline{p}}(n_0)\in\Omega_{TH}&\hbox{if and only
if}& \left(\tilde{R}-\frac{1}{M}\right)\mu_1\leq n_0\leq \tilde{R}M,\hbox{ and,}\\ &&\tilde{R}-\frac{1}{\mu_{n_0+1}}\leq W\left(n_0-1,n_0;\tilde{\underline{p}}(n_0)\right)\leq \tilde{R}.\label{condequilpthrsvr_V3}
\end{eqnarray}
\end{theorem}   

\begin{proof} 
The proof is immediate by \eqref{wtf}, \eqref{neccondequilpthr} and \eqref{condequilpthrsvr_V2}.
\end{proof}

Theorem \ref{condequilpthrsvr} indicates the necessary steps that a finite algorithm must perform in order to identify the pure threshold equilibrium strategies. First, we derive the range of integers which correspond to possible equilibrium pure threshold strategies from \eqref{neccondequilpthr}. Next, for each candidate for equilibrium $\underline{\tilde{p}}(n_0)$, we compute  the waiting time $W(n_0-1,n_0;\underline{\tilde{p}}(n_0))$ solving the corresponding system of equations in \eqref{w0m} and \eqref{wnm} for $0\leq n\leq n_0-1$ and $n+1\leq m\leq n_0$. Finally, if the equilibrium condition stated in \eqref{condequilpthrsvr_V3} is verified for the corresponding value of $W(n_0-1,n_0;\underline{\tilde{p}}(n_0))$, then the corresponding pure threshold strategy $\underline{\tilde{p}}(n_0))$ is an equilibrium. Therefore, the number of equilibria may vary from $0$ to $\lfloor \tilde{R}\;M\rfloor-\lfloor \left(\tilde{R}-\frac{1}{M}\right)\mu_1\rfloor+1$, depending on the monotonicity of $W(n_0-1,n_0;\underline{\tilde{p}}(n_0))$ with respect to $n_0$ and the service rate prescribed at state $n_0+1$, i.e. $\mu_{n_0+1}$.  

In general a pure threshold equilibrium strategy is not unique. In Section \ref{sec-thres}, where we study a specific threshold-type form of the service policy, we construct examples with multiple pure threshold equilibria. 

\section{Mixed threshold strategies}\label{sec-Mixed}
In this section, we extend the analysis to mixed threshold strategies, in which the decision of a customer at the threshold state only is randomized. In general a mixed threshold strategy is determined by a real number $x>0$, as follows. Let $n_0=\lfloor x \rfloor+1$.
Then $p_n=1$ for any $n\leq n_0-2$ and $p_{\lfloor x \rfloor}=x-\lfloor x \rfloor\in(0,1)$. For coherence, we denote these strategies with $\underline{\hat{p}}(x)$ and their corresponding set with 
\begin{equation*}\hat{\Pi}_{TH}=\{\underline{\hat{p}}(x), x\in\Bbb{R}\} \supseteq\Pi_{TH}.
\end{equation*}

%, then a customer's mixed threshold joining strategy is defined as the strategy $\underline{p}\in\Pi_M$ with 

Since there exist multiple pure threshold equilibria, as we show in the following, it is interesting to ask, whether the consideration of  this specific type of randomized strategies which attain the threshold structure, may bridge the equilibrium analysis between pure and mixed threshold strategies, providing an algorithm for the derivation of mixed threshold equilibria relying on the existence of pure threshold equilibria. In many recent works, e.g. \cite{economou_kanta2008} and \cite{economou_manou2013}, there are cases where between consecutive pure threshold equilibria we can derive a mixed solution. In our case, the existence of a mixed threshold equilibrium strategy between consecutive pure threshold equilibria can be identified only numerically, as we show in the following sections. 

Considering the equilibrium conditions in \eqref{condequilRC} and letting $\hat{\Omega}_{TH}$ be the set of mixed threshold equilibrium strategies in the recurrent class, it follows that these differ from the conditions given in \eqref{condequilpthr_V1} for the pure threshold strategies only for $n=n_0-1$, where a potential customer randomizes her decision on joining with probability $ p_{n_0-1}=x-\lfloor x \rfloor$. Specifically, from \eqref{condequilRC}, it follows that the corresponding condition for $n=n_0-1$ in order $\underline{\hat{p}}(x)\in\hat{\Omega}_{TH}$, turns into the following equality 
\begin{equation}\label{condequilmth_limcust}
\tilde{R}-W(n_0-1;\underline{\hat{p}}(x))=0.
\end{equation}

Furthermore, the result of Proposition \ref{st_induction} holds in the framework of mixed threshold equilibria, using a slightly extended coupling argument in order to include the case of mixed threshold strategies, and the proof is quite similar along the same lines under this extended scheme. Specifically, we extend the coupling scheme stated above Proposition \ref{st_induction}, as follows. We assume that each customer tosses a coin  with probability of heads $p_{\lfloor x \rfloor}$ before her arrival to the system, in order to make a decision on joining or not if upon arrival finds $n_0-1$ customers in the system, and, we couple the two systems not only on customer arrivals and service requirements, but also on the predetermined join decision at state $n_0-1$. The monotonicity of $W(n;\underline{\hat{p}}(x))$ in $n$ can be completely proved following analogous arguments as in the proof of Proposition \ref{st_induction}.

The latter result, as in the case of pure threshold strategies, suffices to show that a tagged customer's generalized waiting time $W(n,n+1)$ is non-decreasing in $n$ for any mixed threshold strategy $\underline{\hat{p}}(x)$, and, thus $W(n;\underline{\hat{p}}(x))$ is non-decreasing in $n$ for any $\underline{\hat{p}}(x)$. Finally, due to the monotonic behavior of $W(n;\underline{\hat{p}}(x))$, the equilibrium conditions for $\underline{\hat{p}}(x)\in\hat{\Omega}_{TH}$, can be simplified to a single equation stated in the following Theorem.

\begin{theorem}\label{mtheq}
Given a mixed threshold joining strategy $\underline{\hat{p}}(x)$,
\begin{equation}\label{condequilmth_eq}
\underline{\hat{p}}(x)\in\hat{\Omega}_{TH}\hbox{ if and only if }W\left(\lfloor x \rfloor,\lfloor x \rfloor+1;\underline{\hat{p}}(x)\right)=\tilde{R}. 
\end{equation}
\end{theorem}

% Similarly to the case of pure threshold symmetric equilibria, we define $\hat{\Omega}_{TH}$ as the set of mixed threshold equilibrium strategies in the recurrent class and for , the  equilibrium conditions in  can be rewritten as follows
%
%\begin{equation}
%\label{condequilmthr_V1}\underline{\hat{p}}(x)\in\Omega_{TH}\;\hbox{if and only
%if}\;\left\{\begin{array}[pos]{l}\tilde{R}-W(n_0;\underline{p})\leq0\\\tilde{R}-W(n_0-1;\underline{p})=0\\ \tilde{R}-
%W(n;\underline{p})\geq0,\; \forall n\leq
%n_0(\underline{p})-2\end{array}\right.,
%\end{equation}
%since .

\section{Threshold service rate policy}\label{sec-thres}

In this section, we examine the special case where the service administrator applies a threshold-based service rate policy defined by a positive integer $T$ and two values for the service rate, $\mu_l<\mu_h$, such that the service rate is set to the low value $\mu_l$ when the number of customers in the system is at or below $T$, and to the high value $\mu_h$ when the number of customers exceeds  $T$. We also assume that the service policy $(T,\mu_l,\mu_h)$ is known to all arriving customers. 

In the framework of pure threshold joining strategies, in order to analyze the equilibrium behavior under the threshold-based service policy $(T,\mu_l,\mu_h)$, we consider three cases for $n_0$ with respect to the service threshold $T$. 

Specifically, if $n_0<T$, then for any strategy $\underline{p}$ with $n_0(\underline{p})=n_0$ and, in particular for the pure threshold strategy $n_0$, the service rate is always kept at the low value $\mu_l$, thus, we have an $M/M/1$ queue with finite buffer size $n_0$ and exponentially distributed service times with rate $\mu_l$. In this case, the generalized waiting time is equal to 
\begin{equation}\label{ewtn0belT} W(n,m;\underline{p})=\frac{n+1}{\mu_l},\hbox{ for any }m,\;n \hbox{ with }n<m\leq n_0(\underline{p}),
\end{equation} 
The waiting time of an arriving customer who finds $n$ customers in the system and all follow mixed strategy $\underline{p}$ is equal to

\begin{equation}\label{wtlow}
W(n;\underline{p})=\frac{n+1}{\mu_l},\hbox{ for any }n=0,1,\ldots,n_0-1. 
\end{equation}

When $n_0=T$, the generalized waiting time $W(n,m;\underline{p})$ is still given by \eqref{ewtn0belT} for $n\leq n_0-1$, since the service rate is $\mu_l$. For $n=n_0=T$, we derive from \eqref{wtf} that $W(n_0;\underline{p})=\frac{1}{\mu_h}+W(n_0-1,n_0;\underline{p})$. In this case, a tagged customer who joins and finds $T$ customers in the system will force the service rate to switch to $\mu_h$, thus the residual service time of the customer in service will be exponential with rate $\mu_h$. In summary, the expected waiting time is given by

\begin{equation}\label{wtT}
W(n;\underline{\tilde{p}}(T))=\left\{\begin{array}[pos]{ll}
\frac{n+1}{\mu_l},&n=0,1,\ldots,T-1\\
\frac{1}{\mu_h}+\frac{n+1}{\mu_l},&n=T
\end{array}\right..
\end{equation}

Finally, for $n_0>T$, the service rate switch may occur several times, depending on the evolution of the system length until the tagged customer's departure. In order to derive the generalized waiting time function $W(n,m;\underline{p})$, we solve the system \eqref{w0m}, \eqref{wnm}, which now simplifies to:

For $n=0$:
%\begin{equation}
%\label{w0n0th}W(0,n_0)=\frac{1}{\mu_h},
%\end{equation}
\begin{equation}
\label{w0mhighth}W(0,m)=\frac{1}{\lambda p_m+\mu_h}+\frac{\lambda
p_m}{\lambda p_m+\mu_h}W(0,m+1),\hbox{ for } T+1\leq m\leq n_0,
\end{equation}
\begin{equation}
\label{w0mlowth}W(0,m)=\frac{1}{\lambda p_m+\mu_l}+\frac{\lambda
p_m}{\lambda p_m+\mu_l}W(0,m+1),\hbox{ for }1\leq m\leq T.
\end{equation}

For $1\leq n\leq n_0-1$:
%\begin{equation}
%\label{wnn0th}W(n,n_0)=\frac{1}{\mu_h}+W(n-1,n_0-1),
%\end{equation}
\begin{eqnarray}
\label{wnmhighth}W(n,m)=\frac{1}{\lambda p_m+\mu_h}+\frac{\lambda
p_m}{\lambda p_m+\mu_h}W(n,m+1)+\frac{\mu_h}{\lambda p_m+\mu_h}W(n-1,m-1),\\\nonumber\hbox{ for }\max\{n+1,T+1\}\leq m\leq n_0,
\end{eqnarray}
\begin{eqnarray}
\label{wnmlowth}W(n,m)=\frac{1}{\lambda p_m+\mu_l}+\frac{\lambda
p_m}{\lambda p_m+\mu_l}W(n,m+1)+\frac{\mu_l}{\lambda p_m+\mu_l}W(n-1,m-1),\\\nonumber\hbox{ for } n+1\leq m\leq T,
\end{eqnarray}

%Note that, from the above system of equations we derive the following inequality:
%\begin{equation*}
%W(n,m)=\frac{n+1}{\mu_h},\;\;\hbox{for}\;0\leq n\leq n_0-T-1,\;n+T+1\leq
%m\leq n_0,
%\end{equation*}
%which is intuitive since there is a critical total amount of customers being in the system in order the service rate to be kept at its high value $\mu_h$. Thus, a tagged customer who has $n$ customers ahead will be served under the high service rate, if the total number of customers exceeds $T$ on the same amount.

and the corresponding waiting time function $W(n;\underline{p})$ takes the following form 
\begin{equation}\label{expwtf_th}W(n;\underline{p})=\left\{\begin{array}[pos]{ll}
W(n,n+1),&0\leq n\leq n_0(\underline{p})-1\\
\frac{1}{\mu_h}+W(n_0-1,n_0),& n=n_0(\underline{p})
\end{array}\right.,
\end{equation} 
since $\mu_{n_0+1}=\mu_h$.

As we have already discussed, we proceed with the equilibrium analysis considering only pure threshold joining strategies, i.e. for $\tilde{\underline{p}}(n_0)\in \Pi_{TH}$. We consider the same three cases for $n_0$ as in the derivation of $W(n;\underline{p})$. We summarize the equilibrium analysis in Theorem \ref{condequilpthr}. 

\begin{theorem}\label{condequilpthr}
Given a service policy $(T,\mu_l,\mu_h)$, it follows that:
\begin{itemize}
\item[i.] In the range $\left\{0,1,\ldots,T\right\}$, there exist at most two equilibrium threshold strategies $\underline{\tilde{p}}(\hat{n}_0)$ such that: 

\begin{itemize}
\item[a.] If $\tilde{R}\mu_l\leq T$ and $\tilde{R}\mu_l\in \Bbb{Z}$, then there exist two distinct equilibria with $\hat{n}_0^{(1)}=\tilde{R}\mu_l-1$ and $\hat{n}_0^{(2)}=\tilde{R}\mu_l$. 

\item[b.] If $\tilde{R}\mu_l<T$ and $\tilde{R}\mu_l\not\in\Bbb{Z}$, then $\hat{n}_0=\lfloor \tilde{R}\mu_l \rfloor$ is the unique equilibrium. 

\item[c.] If $T<\tilde{R}\mu_l\leq T+\frac{\mu_l}{\mu_h}$ , then $\hat{n}_0=T=\lfloor\tilde{R}\mu_l\rfloor$ is the unique equilibrium.

\item[d.] If $\tilde{R}\mu_l> T+\frac{\mu_l}{\mu_h}$, then there is no $\underline{\tilde{p}}(n_0)$ equilibrium for $n_0\in\left\{0,1,\ldots,T\right\}$. 

\end{itemize}

\item[ii.]  In the range $\left\{T+1,\ldots\right\}$, a pure threshold strategy $\underline{\tilde{p}}(\hat{n}_0)$ is equilibrium if and only if  
\begin{equation}
\label{condequilpthr_V3} L\leq\hat{n}_0\leq U \hbox{ and } \tilde{R}-\frac{1}{\mu_h}\leq W\left(\hat{n}_0-1,\hat{n}_0;\tilde{\underline{p}}(\hat{n}_0)\right)\leq \tilde{R},
\end{equation}
where $L= \max\{\left(\tilde{R}-\frac{1}{\mu_h}\right)\mu_l,T+1\} $ and $U=\max\{ \tilde{R}\mu_h,T+1\} $.
\end{itemize}
\end{theorem}   

\begin{proof}

{\it{i.}} If $n_0<T$, it follows from \eqref{wtlow} and \eqref{condequilpthrsvr_V3} that: 
\begin{equation}\label{condequillow}
\tilde{\underline{p}}(n_0)\in\Omega_{TH}\hbox{ if and only if }\tilde{R}\mu_l-1\leq
n_0\leq\tilde{R}\mu_l,
%\forall n_0<T\in N:\tilde{R}\mu_l-1\leq
%n_0\leq\tilde{R}\mu_l\Rightarrow \underline{\tilde{p}}(n_0)\in
%\Omega_{TH},
\end{equation}
whereas if all customers follow a pure threshold strategy with $n_0=T$, then from \eqref{wtT} and \eqref{condequilpthrsvr_V3}, we derive:
\begin{equation}\label{condequilT}
\hbox{If } \tilde{\underline{p}}(T)\in\Omega_{TH}\Rightarrow \left(\tilde{R}-\frac{1}{\mu_h}\right)\mu_l\leq
T\leq\tilde{R}\mu_l, \hbox{ and,}
\end{equation}
\begin{equation}\label{condequilT_inv}
\hbox{ for } T=\lfloor \tilde{R}\mu_l \rfloor: if \left(\tilde{R}-\frac{1}{\mu_h}\right)\mu_l\leq
T\leq\tilde{R}\mu_l \Rightarrow \underline{\tilde{p}}(T)\in
\Omega_{TH}.
\end{equation}

Therefore, if $\tilde{R}\mu_l$ is an integer and $\tilde{R}\mu_l < T$, then from \eqref{condequillow} both $\hat{n}_0=\tilde{R}\mu_l-1$ and $\hat{n}_0=\tilde{R}\mu_l$ are equilibria. On the other hand, if $T=\tilde{R}\mu_l$, it follows immediate from \eqref{condequilT_inv} that the pure threshold strategy $\underline{\tilde{p}}(T)$ is equilibrium. Furthermore, for $T=\tilde{R}\mu_l$, $\underline{\tilde{p}}(T-1)$ is also an equilibrium pure threshold strategy, since it satisfies \eqref{condequillow}. Thus, case {\it{i(a)}} follows. 

On the other hand, if $\tilde{R}\mu_l$ is not an integer, we consider the following cases for the $\lfloor \tilde{R}\mu_l \rfloor\leq \tilde{R}\mu_l$.

First, if $\lfloor \tilde{R}\mu_l \rfloor< T$, then the pure threshold strategy $\underline{\tilde{p}}(\hat{n}_0)$ with $\hat{n}_0=\lfloor \tilde{R}\mu_l \rfloor$ satisfies \eqref{condequillow}, and also is the unique integer in this interval since $\lfloor \tilde{R}\mu_l \rfloor-1\leq \tilde{R}\mu_l -1<\tilde{R}\mu_l-\frac{\mu_l}{\mu_h}$. Therefore, $\hat{n}_0=\lfloor \tilde{R}\mu_l \rfloor$ is the unique equilibrium threshold in the range of $\{0,\dots,T-1\}$. 

On the other hand, if  $\lfloor \tilde{R}\mu_l \rfloor=T$, it follows from \eqref{condequilT_inv}, that $\underline{\tilde{p}}(T)$ is an equilibrium in the range of $n_0\leq T$, since $\left(\tilde{R}-\frac{1}{\mu_h}\right)\mu_l\leq\lfloor \tilde{R}\mu_l \rfloor$. In addition,  any threshold strategy with $n_0\leq T-1=\lfloor \tilde{R}\mu_l \rfloor-1<\tilde{R}\mu_l-1$ does not satisfy \eqref{condequillow}, and, thus $\hat{n}_0=T$ is the unique equilibrium in the range of $\{0,\ldots,T\}$, which concludes the proof of {\it{i(b)}}.

Finally, if $\lfloor \tilde{R}\mu_l \rfloor>T+1$, then \eqref{condequillow} does not hold for any $n_0\leq T-1$, and \eqref{condequilT}, \eqref{condequilT_inv} also do not hold. Furthermore, the pure threshold strategy $\underline{\tilde{p}}(T)$ with $T=\lfloor \tilde{R}\mu_l \rfloor$ cannot be an equilibrium since in this case $T>T+1$, which is a contradiction.

Note that, any equilibrium with $\hat{n}_0<T$ coincides with the corresponding equilibrium in \cite{naor}, since under this strategy the system is an $M/M/1$ queue with service rate equal to $\mu_l$. Note that, for $n_0$ in this range, \eqref{condequillow} does not follow from \eqref{neccondequilpthr} by setting $\mu_1=\mu_l$ and $M=\mu_h$ as dictated by the policy $(T,\mu_l,\mu_h)$. Actually, it is a stricter version of \eqref{neccondequilpthr} since for $n_0<T$ the queueing system in steady state will employ only the low service rate, and, thus, the actual value of the upper bound will be equal to $\mu_l$. 

{\it{ii.}} For a pure threshold strategy $\underline{\tilde{p}}(n_0)$ with $n_0> T$, the higher service rate $\mu_h$ will be employed since customers may join in states higher than the service threshold. 

Therefore, for this case $M=\mu_h$, and, thus we can rewrite the range of possible equilibria given in \eqref{neccondequilpthr}, as follows:
\begin{equation}
\label{neccondequilpthr_scase}
\max\{(\tilde{R}-\frac{1}{\mu_h})\mu_l,T+1\}\leq n_0\leq \max\{\tilde{R}\mu_h, T+1\}.
\end{equation}

The latter result provides a range of integers greater than $T$, which correspond to equilibrium pure threshold strategies $\underline{\tilde{p}}(n_0)$, if they also satisfy the condition given in \eqref{condequilpthrsvr_V3} for $\mu_{n_0+1}=\mu_h$, since the corresponding threshold policy is non-decreasing in $n$, and, thus the result of Theorem \ref{condequilpthrsvr} still holds.
\end{proof}

Thus, for any threshold-based service policy there may exist multiple pure threshold equilibria which can be derived explicitly from the parameters in the range $\{0,\ldots,T\}$. 
 
From the algorithmic point of view for threshold strategies with $n_0>T$, candidates for pure threshold equilibrium strategies are finite, and at most $U-L+1$ in number. They can be identified numerically, by solving the system of equation in \eqref{w0mhighth} - \eqref{wnmlowth} for each possible equilibrium $n_0\in[L,U]$, examining whether the solution satisfies \eqref{condequilpthr_V3}. A numerical analysis for the equilibrium pure threshold strategies for varying values of the service reward $R$ is presented in the following section. 

In the framework of mixed threshold joining strategies $\underline
{\hat{p}}(x)$, an interesting question is whether there may exist mixed threshold equilibria between consecutive pure threshold equilibria.

In the range $\{0,\ldots,T\}$, there exist two distinct pure threshold equilibria if and only if the parameters $\tilde{R}$ and $\mu_l$ satisfy the conditions prescribed in case {\it{i(a)}} of Theorem \ref{condequilpthr}. In this case, the pure threshold equilibria are $\hat{n}^{(1)}_0=\tilde{R}\mu_l-1$ and $\hat{n}^{(2)}_0=\tilde{R}\mu_l$, which are consecutive, since $\tilde{R}\mu_l\in\Bbb{Z}$. Considering the generalized waiting time of a joining customer for $n=n_0-1$, it follows from \eqref{wtT}, which is also valid for any mixed threshold strategy $\underline{\hat{p}}(x)$, that 
\begin{equation*}
W(\hat{n}^{(1)}_0,\hat{n}^{(1)}_0+1;\underline{\hat{p}}(x))=W(\hat{n}^{(1)}_0,\hat{n}^{(2)}_0;\underline{\hat{p}}(x))=\frac{\hat{n}^{(2)}_0}{\mu_l}=\frac{\tilde{R}\mu_l}{\mu_l}=\tilde{R}, \hbox{ for any }\underline{\hat{p}}(x).
\end{equation*}
Therefore, any $x\in(\tilde{R}\mu_l-1,\tilde{R}\mu_l)$ when $\tilde{R}\mu_l\in\Bbb{Z}$ determines an equilibrium mixed threshold joining strategy, since the equilibrium condition in \eqref{condequilmth_limcust} is satisfied.

On the other hand, in the range $\{T+1, T+2,\ldots\}$, we can show that there exist at least one mixed threshold equilibrium between two consecutive pure threshold equilibria, under certain conditions. Indeed, for any two consecutive pure threshold equilibria $\hat{n}_0-1,\;\hat{n}_0$ in the range $\{T+1, T+2,\ldots\}$, we consider the mixed threshold strategies $\underline{\hat{p}}(x)$ with $T<\hat{n}_0-1\leq x<\hat{n}_0$ and the function 
\begin{equation}\label{fungwt}
w(x)=W\left(\lfloor x \rfloor,\lfloor x \rfloor+1;\underline{\hat{p}}(x)\right)=W\left(\hat{n}_0-1,\hat{n}_0;\underline{\hat{p}}(x)\right),
\end{equation}
which refers to the equilibrium condition stated in Theorem \ref{mtheq}. Note that, the value of $w(x)$ for any $x\in[\hat{n}_0-1,\hat{n}_0)$ is derived from the solution of the linear system \eqref{w0mhighth} -\eqref{wnmlowth}, setting $m=\hat{n}_0-1$, i.e., 
\begin{equation}\label{gweqn0-1service}
W(0,\hat{n}_0-1)=\frac{1}{\lambda p_{\hat{n}_0-1}+\mu_h}+\frac{\lambda p_{\hat{n}_0-1}}{\lambda p_{\hat{n}_0-1}+\mu_h}W(0,\hat{n}_0), \hbox{ for } n=0, 
\end{equation}
  
\begin{equation}\label{gweqn0-1}
W(n,\hat{n}_0-1)=\frac{1}{\lambda p_{\hat{n}_0-1}+\mu_h}+\frac{\lambda p_{\hat{n}_0-1}}{\lambda p_{\hat{n}_0-1}+\mu_h}W(\hat{n}_0-1,\hat{n}_0)+\frac{\mu_h}{\lambda p_{\hat{n}_0-1}+\mu_h}W(n-1,\hat{n}_0-2), 
\end{equation}
for $n=1,\ldots,\hat{n}_0-1$, where $p_{\hat{n}_0-1}=x-(\hat{n}_0-1)$ is the join probability at state $\hat{n}_0-1$.

Since $x$ varies only $p_{n_0-1}$, and, thus, the coefficients of the linear system in \eqref{w0mhighth} -\eqref{wnmlowth} with \eqref{gweqn0-1service}, \eqref{gweqn0-1} defined for any different value of $x$, we can show that its solution as a function of $x$ is continuous, and, thus $w(x)$ is also continuous in $[\hat{n}_0-1,\hat{n}_0]$. 

Moreover, we assume that $W(\hat{n}_0-2,\hat{n}_0-1;\underline{\tilde{p}}(\hat{n}_0-1))>\tilde{R}-\frac{1}{\mu_h}$ and $W(\hat{n}_0-1,\hat{n}_0;\underline{\tilde{p}}(\hat{n}_0))<\tilde{R}$, which also hold for the corresponding mixed threshold strategies, since by their definition, mixed threshold strategies coincide with the corresponding pure threshold ones for any $x\in\Bbb{Z}$. Therefore, the existence of a solution of $w(x)=\tilde{R}$ in $(\hat{n}_0-1,\hat{n}_0)$ can be proved, applying Bolzano's Theorem for $w(x)$ in the interval $[\hat{n}_0-1,\hat{n}_0]$. Indeed, we derive that
\begin{equation*}
h(\hat{n}_0)=W\left(\hat{n}_0-1,\hat{n}_0;\underline{\hat{p}}(\hat{n}_0)\right)=W(\hat{n}_0-1,\hat{n}_0;\underline{\tilde{p}}(\hat{n}_0))<\tilde{R}.
\end{equation*}
Moreover, 
\begin{equation*}
h(\hat{n}_0-1)=W\left(\hat{n}_0-1,\hat{n}_0;\underline{\hat{p}}(\hat{n}_0-1)\right),
\end{equation*}
which refers to the expected waiting time of tagged customer which has $\hat{n}_0-1$ customers upfront and all customers follow $\underline{\hat{p}}(\hat{n}_0-1)$, which prescribes at new arrivals to balk. Since the only possible transition is for the customer in service to leave, it follows that 
\begin{eqnarray*}
h(\hat{n}_0-1)=W\left(\hat{n}_0-1,\hat{n}_0;\underline{\hat{p}}(\hat{n}_0-1)\right)=\frac{1}{\mu_h}+ W(\hat{n}_0-2,\hat{n}_0-1;\underline{\hat{p}}(\hat{n}_0-1))\\=\frac{1}{\mu_h}+W(\hat{n}_0-2,\hat{n}_0-1;\underline{\tilde{p}}(\hat{n}_0-1))>\frac{1}{\mu_h}+\tilde{R}-\frac{1}{\mu_h}=\tilde{R},
\end{eqnarray*}
which completes the proof. Note that, under the mixed threshold strategy $\underline{\hat{p}}(\hat{n}_0-1)$, $\left(\hat{n}_0-1,\hat{n}_0\right)$ is a transient state.

On the other hand, if $W(\hat{n}_0-2,\hat{n}_0-1;\underline{\tilde{p}}(\hat{n}_0-1))\geq\tilde{R}-\frac{1}{\mu_h}$ or $W(\hat{n}_0-1,\hat{n}_0;\underline{\tilde{p}}(\hat{n}_0))\leq\tilde{R}$, the existence of mixed threshold equilibria depend on the monotonicity of $w(x)$, and, there may be extreme cases where none of the mixed threshold strategies defined for $x\in(\hat{n}_0-1,\hat{n}_0$ cannot be an equilibrium.

In order to investigate the derivation of mixed threshold equilibria, especially those between consecutive pure threshold equilibria, we perform several numerical examples deriving $w(x)=W\left(\lfloor x \rfloor+1,\lfloor x \rfloor+1+1;\underline{\hat{p}}(x)\right)$ for $x>T$ in the following section. These numerical results showed that $w(x)$ is left continuous and decreasing in $x$, and, thus there exist a unique mixed threshold equilibrium strategy between consecutive pure threshold equilibria. On the other hand, there may also exist mixed threshold equilibria which are not characterized by consecutive pure threshold equilibria, as we show in the following section.      
    
\section{Computational Results}\label{sec-comp}
In this section we explore the number and distribution of pure and mixed threshold equilibria as a function of the   service reward $R$, and relate it to the form of the delay function. For concreteness we consider the case of a $T$ threshold service policy, as in Section \ref{sec-thres}. Under this service policy, the threshold equilibrium strategies are determined according to  Theorem \ref{condequilpthr} and the discussion in Section \ref{sec-thres}.

We analyze a case study for  a system with $\lambda=3,\;\mu_l=2,\;\mu_h=5$, service threshold $T=23$, service reward $R$ varying between $8$ and $13$ and  waiting cost  $C=1$, thus, $\tilde{R}=R$.

The pure threshold equilibrium strategies are presented in Table \ref{ethres_val}, where the second and third columns contain the equilibria in the ranges $\{0,\ldots,T\}$ and $\{T+1,\ldots\}$, respectively. The last two columns contain the lower and upper bound for possible equilbrium strategies above $T$, $L=\max\{(R-\frac{1}{\mu_h})\mu_l,T+1\}$ and $U=R\mu_h$, as in \eqref{neccondequilpthr_scase}.

\begin{table}[H]
\begin{center}
\begin{tabular}{|c|c|ccc|}
  \hline
  &\multicolumn{4}{|c|}{$\underline{\tilde{p}}(n_0)\in\Omega_{TH}$}\\
  \hline
  $R$&$n_0\leq T$&$n_0>T$&$L$&$U$\\
  \hline
  8.0&   15, 16&    $-$ & 24 & 40\\
  8.15&   16&    $26,\ldots,32$ & 24 & 40.75\\
  8.5&   16, 17&    $25,\;36,37$& 24 & 42.5\\
  9.5&   18, 19&    45& 24 & 47.5\\
  13&   $-$&    64& 25.6 & 65\\
  \hline
\end{tabular}
\end{center}
\caption{Threshold values $n_0$ of pure threshold equilibria as a function of $R$, for $\lambda=3,\;\mu_l=2,\;\mu_h=5,\;T=23,\;C=1$.}\label{ethres_val}
\end{table}

From Table \ref{ethres_val}, we first observe that the number of pure threshold  equilibria varies with $R$. For low or high values of $R$ the equilibrium is unique, whereas for intermediate values  there exist multiple equilibria with threshold values either below or above the service threshold $T$.  For example, for  $R=8.5$  there are two  equilibria with threshold values $\hat{n}_0=16$ and $\hat{n}_0=17$ below the service threshold $T=23$ and multiple above $T$.

The pure threshold equilibria below the service threshold $T$ are as characterized in Theorem \ref{condequilpthr} and their number varies between 0 and 2. For values above $T$ the theorem does not uniquely determine their number but only their range and the condition they satisfy. We observe that in general there are multiple pure threshold equilibria in the range between $L$ and $U$ with values distributed in one or two intervals of successive integers. Furthermore, there is no general conclusion as to how sharp are the bounds $L, U$, since there are cases where their range is either narrow or wide compared to the actual distribution of equilibria inside.

We can obtain further insights on the existence of a single or multiple equilibria if we analyze the behavior of $W(n_0-1,n_0;\underline{\hat{p}}(n_0))$, which corresponds to the expected delay of the marginal customer who decides to enter when all other customers follow strategy $n_0$ and there are already $n_0-1$ customers present in the system. A graph of this function is presented in Figure \ref{wtsensR} for values of $n_0$ between $0$ and $40$.

\begin{figure}[H]
\centering
\includegraphics[scale=0.75]{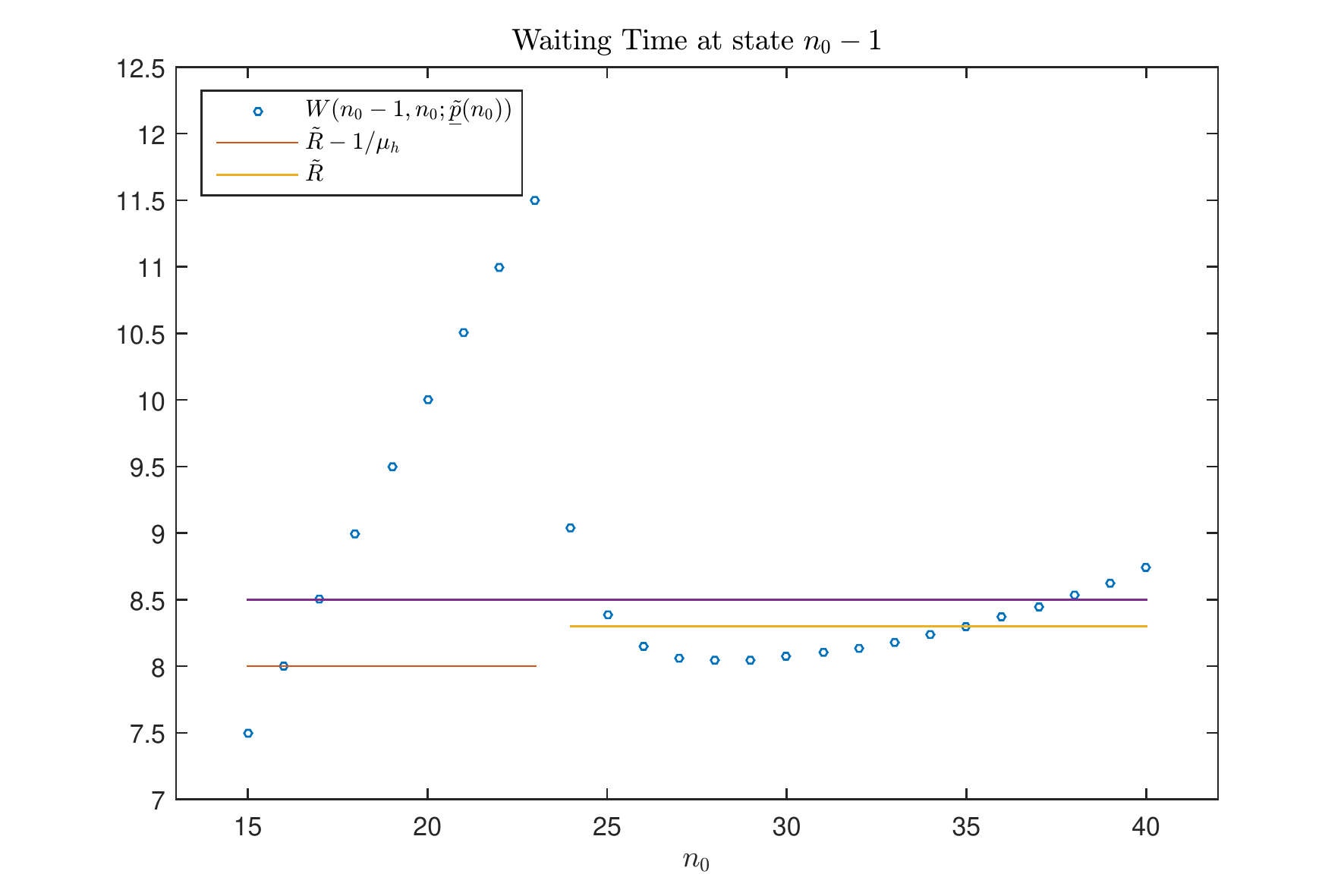}
\caption{Waiting time at state $n_0-1$ as a function of $n_0$, for $\lambda=3,\;\mu_l=2,\;\mu_h=5,\;T=23$, and equilibrium range for $\tilde{R}=8.5$.}\label{wtsensR}
\end{figure}

The behavior of  $W(n_0-1,n_0;\underline{\hat{p}}(n_0))$ explains why for some values of $R$ the pure threshold equilibria are distributed in two separate integer intervals. 

Note that, in contrast to Proposition \ref{st_induction}, which implies that $W(n-1,n;\underline{\hat{p}}(n_0))$ is shown to be increasing in $n$ for a fixed $n_0$, the marginal delay function  $W(n_0-1,n_0;\underline{\hat{p}}(n_0))$  which is involved in the equilibrium condition is not generally monotone in $n_0$. For  $n_0\leq T$ the service rate is always at the low value $\mu_l$ thus the marginal customer's delay is increasing in $n_0$. On the other hand when $n_0$ is high, the service rate is almost always at the high value, and the switching effect is negligible, thus, the delay is also increasing in $n_0$, due to the increasing number of customers the tagged customer encounters. Finally, fo values of $n_0$ close to the service switch threshold $T$ the delay function has a range of decreasing values. This happens because as $n_0$ increases, the proportion of the time that the server works under the fats rate also increases and as a result the delay of the entering customer is reduced. As $n_0$ increases even further, the additional load imposed on the system by the entering customers exceeds the benefit due to the faster service rate and the delay function start increasing again.

We finally consider mixed threshold equilibria $\underline{\hat{p}}(x)$. The equilibrium condition is given in Theorem \ref{mtheq}, while in Section \ref{sec-thres} it was shown that for the threshold service strategy between any two successive pure threshold equilibria there exists a mixed threshold equilibrium strategy. In Figure \ref{wtmixedsensR} we present the graph of the marginal customer's delay function $W(\lfloor x \rfloor,\lfloor x \rfloor+1;\underline{\hat{p}}(x))$  for $x\in(24,39]$. The graph shows how the delay function presented in Figure \ref{wtsensR} for pure threshold strategies is extended to the mixed threshold case. The mixed threshold equilibrium strategies correspond to the points of intersection of the graph of $W(\lfloor x \rfloor,\lfloor x \rfloor+1;\underline{\hat{p}}(x))$ with the horizontal line at $R$.

\begin{figure}[H]
\centering
\includegraphics[scale=0.75]{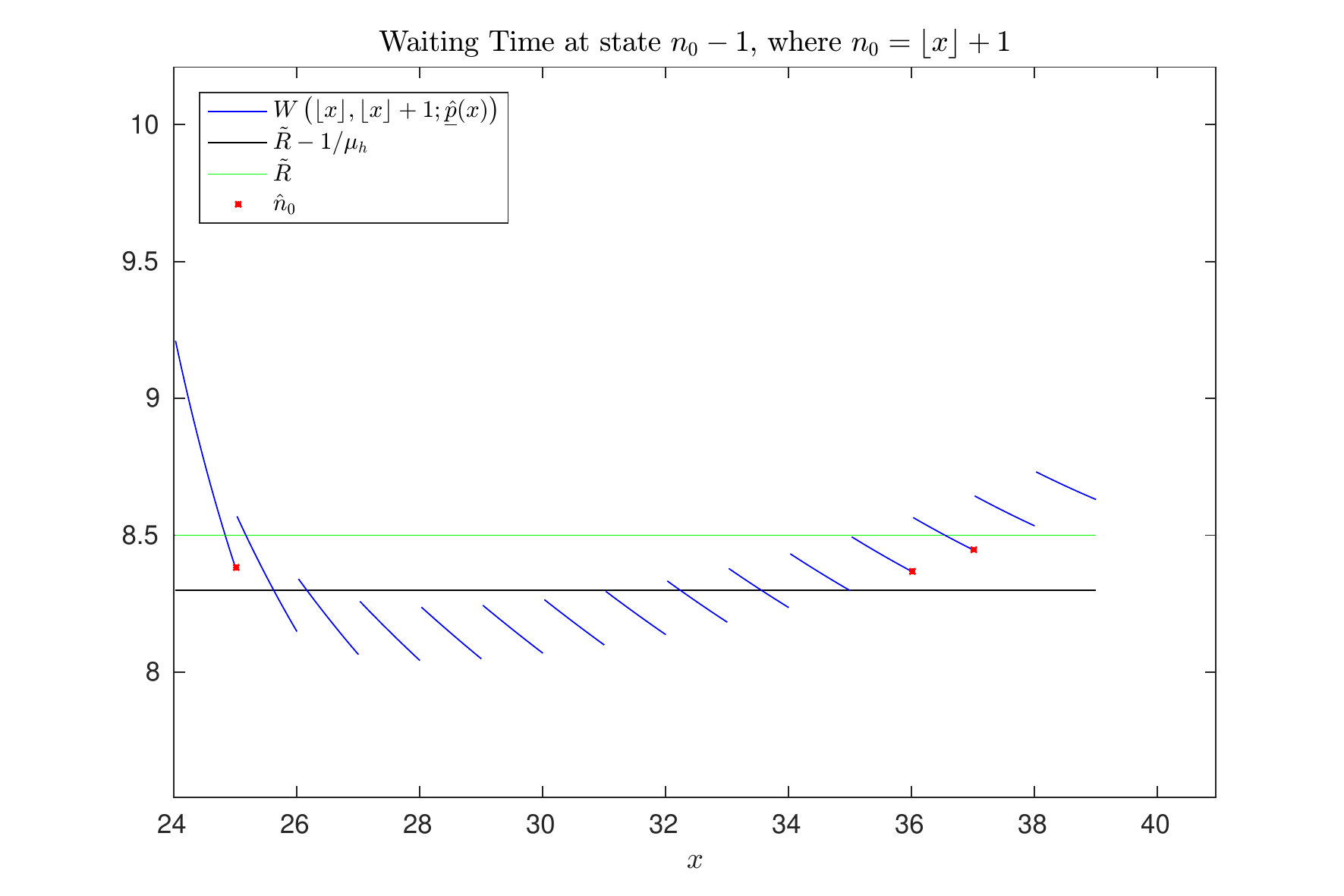}
\caption{Waiting time at state $\lfloor x \rfloor$ as a function of $x$, for $\lambda=3,\;\mu_l=2,\;\mu_h=5,\;T=23$, and equilibrium range for $\tilde{R}=8.5$.}\label{wtmixedsensR}
\end{figure}

We first observe that the delay function is piecewise decreasing in intervals between successive integers and has discontinuities at integer values where it is left-continuous. This behavior is expected. Consider $x$ varying in the interval $[i, i+1)$. For all such $x$ an arriving customer who finds $i$ customers waiting and joins is better off in terms of delay if $x$ is higher, since in this case the server will spend more time in the high mode on average. On the other hand, when $x$ becomes equal to $i+1$, the delay function has a jump due to the one additional customer present in the system. 

The graph of the delay function explains why there exists a mixed threshold equilbrium strategy between two pure threshold equilibria, as discussed in Section \ref{sec-thres}. However it also shows that there are cases of mixed threshold equilibrium strategies where the corresponding integers are not both equilibria. Indeed,  the  pure threshold strategy with  $n_0=25$ is an equilibrium, while for  $n_0=26$ it is not. However there exists a mixed threshold equilibrium $x\in (25,26)$.   

\section{Conclusion and Extensions}
\label{sec-Summary}
\noindent In this paper, we considered the problem of customer equilibrium joining behavior in an M/M/1 queue with
dynamically adjusted non-decreasing service rate with respect to the level of congestion and full information. We examined the observable case of this model, where each arriving customer is aware of the service policy and the current queue length.

We have proved that strategies where balking is not an option cannot be equilibria. For strategies where balking is prescribed at a finite state, we showed that we can restrict attention to policies for which equilibrium conditions hold only for states in the recurrent class. The equilibrium analysis was performed in the class of pure or mixed threshold strategies, i.e. strategies where customers join if and only if they find less customers in the system than a certain threshold value with possible randomization at the threshold state. For this class, we have proved that the expected sojourn time is non-decreasing in the number of observed customers upon arrival, and, thus, the equilibrium  condition depends only on its corresponding value at the threshold state. 

Furthermore, we have derived upper and lower bounds for the pure threshold equiibria, which implies that they can be determined using a finite search algorithm. Also, we have analyzed the special case of a service control policy, based on a single switching threshold and two values of service rate. For this policy, we showed further structure for the equilibrium strategies, analytically and numerically. 

The general problem of customer strategic behavior for joining an observable queueing system with dynamic adjusted service rate can be extended in several directions. In terms of the form of the service control policy, one might consider an alternative where instead of increasing the service speed with congestion, one or more standby servers are activated. Such a policy is more relevant in situations where the service is provided by human workers. There are also several extensions with respect to the level of available information in both directions. In this paper, we have examined the case where arriving customers are fully informed on the service policy and the current queue length upon arrival. However, one may also assume that arriving customers are not informed on the exact number of customers in the system but instead they are given a range. In terms of the available information on service, one could consider the case where the service threshold is not announced in an observable system but customers observe the queue length and the status of the server, and thus they may estimate the service threshold. Finally, a more challenging problem would be to consider the capacity problem where the service  parameters $\mu_l, \mu_h$ and $T$ are considered as control decisions by the system manager, incorporate relevant costs for them and solving for the optimal control policy taking into account the customer response in equilibrium.
%

%
%\bibliography{ref_1,ref_equil,refs}
%\chapter{Βιβλιογραφία}

\end{document}